\documentclass[12pt]{article}

\usepackage{amssymb}
\usepackage{amsthm}
\usepackage{amsmath}
\usepackage{graphicx}
\usepackage{fullpage}
\usepackage{color}
\usepackage{enumerate}
 \numberwithin{equation}{section}
\usepackage{booktabs}
\allowdisplaybreaks

\usepackage[pdftex]{hyperref}
 \usepackage{mathtools}
 \mathtoolsset{showonlyrefs}
 \usepackage[nocompress]{cite}
\usepackage{mathrsfs}
\theoremstyle{plain}
\newtheorem{thm}{Theorem}[section]
\newtheorem{cor}[thm]{Corollary}

\newtheorem{lem}[thm]{Lemma}
\newtheorem{prop}[thm]{Proposition}

\theoremstyle{definition}

\newtheorem{ex}[thm]{Example}

\theoremstyle{remark}

\newtheorem{rem}[thm]{Remark}

\newcommand{\N}{\mathbb{N}}
\newcommand{\R}{\mathbb{R}}

\newcommand{\Pee}{{\cal P}}
\newcommand{\Jay}{\mathscr{J}}

\newcommand{\bp}{\begin{proof}[\ensuremath{\mathbf{Proof}}]}
\newcommand{\bs}{\begin{proof}[\ensuremath{\mathbf{Solution}}]}
\newcommand{\ep}{\end{proof}}
\newcommand{\be}{\begin{equation}}
\newcommand{\ee}{\end{equation}}

\begin{document}

\title{Evolution of mixed strategies in monotone games }

\author{Ryan Hynd\footnote{Department of Mathematics, University of Pennsylvania.  Partially supported by NSF award DMS-1554130.}}

\maketitle 

\begin{abstract}
We consider the basic problem of approximating Nash equilibria in noncooperative games. For monotone games, we design continuous time flows which converge in an averaged sense to Nash equilibria.  We also study mean field equilibria, which arise in the large player limit of symmetric noncooperative games. In this setting, we will additionally show that the approximation of mean field equilibria is possible under a suitable monotonicity hypothesis.  
\end{abstract}

\section{Introduction}
We begin by recalling a noncooperative game in which players labeled $j=1,\dots, N$ each have finitely many choices.  For definiteness, let us suppose that each player selects an action among the first $m$ natural numbers and that each player is unaware of the others' selections.  If player $i$ selects $s_i\in \{1,\dots, m\}$ for $i=1,\dots, N$, player $j$'s cost is a number
$$
f_j(s_1,\dots,s_N).
$$
Each player seeks to have as small of a cost as possible.  However, each player's cost depends on the other players' actions.

\par This type of game leads naturally to the notion of a Nash equilibrium. This is an $N$-tuple $s=(s_1,\dots,s_N)$ for which 
$$
f_j(s)\le f_j(t_j,s_{-j})
$$
for all $t_j=1,\dots, m$ and $j=1,\dots, N$. Here we have written 
$$
(t_j,s_{-j})=(s_1,\dots,s_{j-1},t_j,s_{j+1},\dots,s_N).
$$
Note in particular that no player can pay a smaller cost by making a unilateral change. Simple examples can be found in which Nash equilibria do not exist. However, Nash showed such equilibria exist if mixed strategies are allowed \cite{MR43432,MR31701}. 

\par For any $m\in \N$, we will denote the standard $m$-simplex as
$$
\Delta_{m}=\left\{z\in \R^m: z_j\ge 0,\; \sum^{m}_{j=1}z_j=1\right\}.
$$ 
A {\it mixed strategy} for player $i$ is an element 
$x_i=(x_{i,1},\dots, x_{i,m})\in \Delta_{m},$
which corresponds to player $i$ choosing action $j\in \{1,\dots, m\}$ with probability $x_{i,j}$. If player $i$ selects the mixed strategy $x_i\in \Delta_{m}$ for each $i=1,\dots, N$, player $j$'s expected cost is defined to be
\be\label{FiniteGameFeye}
F_j(x_1,\dots, x_N)=\sum^{m}_{s_1=1}\cdots\sum^{m}_{s_N=1} f_j(s)x_{1,s_1}\cdots x_{N,s_N}.
\ee
We'll also say $x_i\in \Delta_{m}$ is a {\it pure strategy} if one of the entries of $x_i$ is equal to $1$. 

\par We can extend the definition of Nash equilibria given above to incorporate mixed strategies as follows. A {\it Nash equilibrium} is an $N$-tuple $x=(x_1,\dots,x_N)$
for which 
$$
F_j(x)\le F_j(y_j,x_{-j}) 
$$
for each $y_j\in \Delta_{m}$ and $j=1,\dots, N$. Here $(y_j,x_{-j})$ is defined analogously to $(t_j,s_{-j})$ above.  As with Nash equilibria for pure strategies, no player can improve her expected cost by deviating from her current choice.  In this note, 
we will discuss the possibility of approximating Nash equilibrium for this type and more general types of games. 
\subsection{Previous work}
The existence of a Nash equilibrium for the game discussed above follows from an application of Brouwer's fixed point theorem.  Since proofs of Brouwer's fixed point theorem are nonconstructive, it seems unlikely that there would be an easy way to approximate Nash equilibria in general. This problem has been examined at length, and its complexity has been categorized as being equivalent to finding a fixed point in the conclusion of Brouwer's theorem  \cite{MR2506524,MR2510264,MR2536129}. In particular, there is no known efficient algorithm for approximating Nash equilibria. 

\par Nevertheless, there is one class of games in which approximation is at least theoretically feasible. These games are called {\it monotone}. In the context described above, their expected cost functions $F_1,\dots, F_N$ satisfy 
\be\label{initFmonotone}
\sum^N_{j=1}(F_j(x)+F_j(y))\ge \sum^N_{j=1}(F_j(x_j,y_{-j})+F_j(y_j,x_{-j}))
\ee
for $x,y\in \Delta_m^N$.  For example, any two-player zero-sum game satisfies this monotonicity condition (see Corollary \ref{twoPlayerZeroSumMon} below). This condition additionally extends more generally to cost functions $F_1,\dots, F_N$ which are not necessarily of the form \eqref{FiniteGameFeye}.  We also note there have been several recent studies on monotone games \cite{gangbo2021mean,alpar2021mean,MR3074829,MR3225957,MR3505343,MR3907835,MR3926912,MR3939746,MR4334439,MR4334453,MR4350454,MR3023769}.

\par In prior joint work  \cite{awi2020continuous}, we argued that if the game is monotone, then for any $x^0\in  \Delta_m^N$, there is an absolutely continuous path $u:[0,\infty)\rightarrow \Delta_m^N$
such that 
\be\label{classicalFlow}
\begin{cases}
\dot u_j(t)+\partial_{x_j}F_j(u(t))\ni 0
\\
 u_j(0)=x^0_j
\end{cases}
\ee
for each $j=1,\dots, N$. Here
\be\label{classicalSubdiff}
\partial_{x_j}F_j(x)=\Big\{z\in \R^{m}:F_j(y_j,x_{-j})\ge F_j(x)+z\cdot (y_j-x_j) \;\text{for $y_j\in \Delta_{m}$}\Big\},
\ee
and the dot `$\cdot$' denotes the standard dot product on $\R^{m}$.   This well-posedness can be established using the theory of semigroups of maximal monotone operators on a Hilbert space as detailed in the monograph by Br\'ezis \cite{MR0348562} and 
originating with the seminal works of Kato \cite{MR226230,MR0271782} and K$\bar{\text{o}}$mura \cite{MR216342}.
 
\par We also considered the {\it Ces\`aro mean} of $u$
$$
\frac{1}{t}\int^t_0u(s)ds.
$$
Let us suppose for the moment that this mean converges to $x$ as $t\rightarrow\infty$. In view of \eqref{initFmonotone} and \eqref{classicalFlow}, 
\begin{align}
\sum^N_{i=1}\left(F_i(z)-F_i(u_i(s),z_{-i})\right) &\ge \sum^N_{i=1}\left(F_i(z_i,u_{-i}(s))-F_i(u(s))\right)\\
&\ge -\sum^N_{i=1}\dot u_i(s)\cdot (z_i-u_i(s))  \\
&= \frac{d}{ds}\sum^N_{i=1}\frac{1}{2}|u_i(s)-z_i|^2
\end{align}
for $z\in \Delta_m^N$.  Integrating from $s=0$ to $s=t$ and dividing by $t$ gives 
$$
\sum^N_{i=1}\left(F_i(z)-F_i\left(\frac{1}{t}\int^t_0u_i(s)ds,z_{-i}\right)\right) \ge\sum^N_{i=1}\frac{1}{2t}\left(|u_i(s)-z_i|^2-|u_i^0-z_i|^2\right).
$$

\par As $u(t)\in \Delta_m^N$ is bounded, we can send $t\rightarrow\infty$ to find 
$$
\sum^N_{i=1}\left(F_i(z)-F_i\left(x_i,z_{-i}\right)\right)\ge 0.
$$
Choosing $z=(y_j,x_{-j})$ for $y_j\in \Delta_m$ would then lead to 
$$
F_j(y_j,x_{-j})-F_j(x)\ge0.
$$
That is, $x$ is a Nash equilibrium.  Of course it remains to be shown that the Ces\`aro mean of $u$ converges. This follows from a theorem due Baillon and Br\'ezis \cite{MR394328}. The goal of this study is to identify a general setting in game theory for which we can apply this result.

\subsection{A general setting}
In what follows, we will study a general version of the noncooperative game detailed above.  To this end, we will consider a separable  Banach space $X$ with continuous dual space $X^*$ and write 
$$
\mu(x)=\langle \mu,x\rangle
$$
for $\mu\in X^*$ and $x\in X$. Let us suppose $K_1,\dots, K_N\subset X^*$ are each nonempty, convex, weak* compact, and set
$$
K=K_1\times \cdots \times K_N.
$$

\par We will study collections of $N$ functions $F_1,\dots, F_N: K\rightarrow \R$ which are weak* continuous and satisfy  
\be\label{FeyeConvexityCondition}
K_j\ni \nu_j\mapsto F_j(\nu_j,\mu_{-j})\text{ convex}
\ee
for each $\mu\in K$ and $j=1,\dots, N$. We'll say $\mu\in K$ is a {\it Nash equilibrium} of $F_1,\dots, F_N$ provided that 
\be
F_j(\mu)\le F_j(\nu_j,\mu_{-j})\;\text{for all $\nu_j\in K_j$ and $j=1,\dots, N$.}
\ee
Later in this note, we will briefly recall how to justify the existence of a Nash equilibria. 

\par The prototypical scenario of interest is when $X=C(S)$ for a compact metric space $S$ and
$$
K_1=\cdots=K_N=\Pee(S).
$$
Here $\Pee(S)$ is the collection of Borel probability measures on $S$. We recall $X^*$ is isometrically isomorphic to $M(S)$, the collection of Radon measures on $S$ equipped with the total 
variation norm. Moreover,  $\Pee(S)\subset M(S)$ is convex and weak* compact. Note that if $f_j: S^N\rightarrow \R\;\text{ is continuous}$,
then 
\be\label{GameTheoryFeye}
F_j(\mu)=\int_{S^N}f_j(s)d\mu_1(s_1)\cdots d\mu_N(s_N)
\ee
is weak* continuous on $\Pee(S)^N$ for $j=1,\dots, N$. Moreover, $F_j$ clearly satisfies \eqref{FeyeConvexityCondition}.

\par These objects relate to game theory as follows. The set $S$ represents an action space for players $1,\dots, N$ in a noncooperative game.  An element $\mu_j\in \Pee(S)$ constitutes a mixed strategy for player $j$; that is, player $j$ chooses from a given collection of actions $A\subset S$ with probability $\mu_j(A)$.  Of course, $\mu_j=\delta_{s_j}$ is a pure strategy: player $j$ always select action $s_j\in S$. The value $f_j(s)$ represents player $j$'s cost if the players collectively opt for action $s=(s_1,\dots, s_N)\in S^N$. And $F_j(\mu)$ indicates player $j$'s expected cost if players $1,\dots, N$ respectively select the mixed strategies $\mu_1,\dots, \mu_N$.  Note than when $S$ is finite, this example corresponds to the $N$-player noncooperative game considered above.
\subsection{Approximation result} 
We aim to use a flow along the lines of \eqref{classicalFlow} to approximate Nash equilibria for $F_1,\dots, F_N$ in the general setting outlined above. An important detail in \eqref{classicalFlow} that we made use of is the natural embedding 
$$
\Delta_{m}\subset \R^{m}.
$$
Here $ \R^{m}$ is a Hilbert space with the usual dot product. With this goal in mind, we will employ 
$$
\text{a centered, nondegenerate Gaussian measure $\eta$ on $X$.}
$$
Recall this means $\eta$ is a Borel probability measure on $X$ such that the push forward of $\eta$ by any nonzero element of $X^*$ is a centered, nondegenerate Gaussian measure on $\R$. It turns out that $X^*\subset L^2(X,\eta)$, and we will see that the Hilbert space 
$$
H=\text{the closure of $X^*$ in $L^2(X,\eta)$}
$$
will play the role of $\R^m$ with the dot product for the flow we present below. 

\par  In analogy with \eqref{classicalSubdiff}, we define 
\be\label{NewSubdiff}
\partial_{\mu_j}F_j(\mu)=\Big\{x\in X:F_j(\nu_j,\mu_{-j})\ge F_j(\mu)+\langle\nu_j-\mu_j,x\rangle \;\text{for $\nu_j\in K_j$} \Big\}
\ee
for $\mu\in K$. Note that $\mu\in K$ is a Nash equilibrium if and only if 
\be
0\in \partial_{\mu_j}F_j(\mu)\;\text{for all $j=1,\dots, N$.} 
\ee
In addition, we'll say that $F_1,\dots, F_N$ is {\it monotone} provided 
\be\label{newFmonotone}
\sum^N_{j=1}\langle \mu_j-\nu_j,x_j-y_j\rangle \ge 0
\ee
whenever $x_j\in \partial_{\mu_j}F_j(\mu)$ and $y_j\in \partial_{\mu_j}F_j(\nu)$ for $j=1,\dots, N$. We will also see in Proposition  \ref{OneMonImpliesOther} that the aforementioned type of monotonicity \eqref{initFmonotone} is a special case of the notion just introduced. 

\par In the following theorem, we will make use of the $L^2(X,\eta)$ inner product $(\cdot,\cdot)$, a linear mapping $\Jay: H\rightarrow X$ which satisfies
\be
(\mu,\nu)=\langle \mu,\Jay\nu\rangle\;\text{ for $\mu,\nu\in X^*$},
\ee
and the set 
\be\label{DeeSet}
{\cal D}=\left\{\mu\in K: \bigcap^N_{j=1}\Jay^{-1}\left(\partial_{\mu_j}F_j(\mu)\right)\neq\emptyset\right\}.
\ee
\begin{thm}\label{ThmOne}
Suppose $F_1,\dots, F_N$ satisfies \eqref{FeyeConvexityCondition} and \eqref{newFmonotone} and that $\mu^0\in {\cal D}$.  There is a unique absolutely continuous $\xi:[0,\infty)\rightarrow H^N$ with $\xi(t)\in {\cal D}$ for each $t\ge0$ and 
\be
\begin{cases}
\Jay\dot \xi_j(t)+\partial_{\mu_j}F_j(\xi(t))\ni 0\;\text{a.e. $t\ge 0$}
\\
\xi_j(0)=\mu^0_j
\end{cases}
\ee
for each $j=1,\dots, N$. Moreover, $\xi$ is Lipschitz continuous and 
$$
\frac{1}{t}\int^t_0\xi(s)ds
$$
converges weak* to a Nash equilibrium of $F_1,\dots, F_N$ as $t\rightarrow\infty$.
\end{thm}


\par We will also prove a related approximation theorem for symmetric games.  A
prototypical example occurs when 
$F_1,\dots, F_N$ is defined via \eqref{GameTheoryFeye} with
$$
f_i(s)=f\left(s_i,\frac{1}{N-1}\sum_{j\ne i}\delta_{s_j}\right)
$$
for $i=1,\dots, N$ and some continuous $f:S\times \Pee(S)\rightarrow \R$.  It turns out that $F_1,\dots, F_N$ has a symmetric Nash equilibrium $(\mu^N,\dots, \mu^N)\in \Pee(S)^N$.  Furthermore, when $N\rightarrow\infty$, $(\mu^N)_{N\in \N}$ has a subsequence which converges weak* to some $\mu$ that satisfies 
$$
\int_{S}f(s,\mu)d\mu(s)\le \int_{S}f(s,\mu)d\nu(s)
$$
for each $\nu \in \Pee(S)$ (as explained in Chapter 4 of \cite{DanNotes}). Such a $\mu$ is called a {\it mean field equilibrium}. Finding mean field equilibria is a basic problem in the theory of mean field games \cite{MR3967062,DanNotes,MR3752669,MR3195844} and we will informally refer to this example as a {\it static mean field game}.   In Theorem \ref{ThmTwo} below, we will employ a simpler version of the flow described in Theorem \ref{ThmOne} to approximate symmetric and mean field equilibria. 
\\
\par Most approximation results for Nash equilibria under monotonicity constraints, such as the ones verified in \cite{MR3703507, MR2948917,  MR4090377,MR3108427,MR3361444,MR3337989,MR4205673,MR3225845}, involve discrete time flows. The first study that used a continuous time flow to approximate Nash equilibria in monotone games set in finite dimensions was initiated by Fl\aa m \cite{MR1201625}.  
In our prior work \cite{awi2020continuous},  we extended Fl\aa m's work to Hilbert spaces and highlighted the role of the Ces\`aro mean. The contribution of this paper is in verifying that theoretical approximation can be obtained with a continuous time flow for monotone games set in dual Banach spaces.

\par  This paper is organized as follows. In section \ref{prelimSect}, we will recall some basic facts about Gaussian measures. Next, we will discuss general $N$-player games in section \ref{NplayGameSect} and prove Theorem \ref{ThmOne}.    Then in section \ref{SymmetricMFGsect}, we will show how to approximate equilibria in symmetric and static mean field games provided that the appropriate monotonicity hypothesis is in place.  In the appendix, we will show how our general theory reduces to the type of game discussed at the beginning of this introduction and work out an explicit example to illustrate why we can't expect to have better than convergence in the sense of the Ces\`aro mean.

\section{Preliminaries}\label{prelimSect}
As in the introduction, we will suppose $X$ is a separable Banach space over $\R$ with norm $\|\cdot \|$ and denote the space of continuous linear functionals $\mu: X\rightarrow \R$ as $X^*$. We will also express the dual norm as 
$$
\|\mu\|_*=\sup\{|\mu(x)|: \|x\|\le 1\}.
$$
Note that since $X$ is separable, the weak* topology on $X^*$ is metrizable. In particular, $\mu^k\rightarrow \mu$ weak* whenever
$$
\lim_{k\rightarrow\infty}\mu^k(x)=\mu(x)
$$
for all $x\in X$. It will also be important for us to recall that the unit ball $\{\mu\in X^*:\|\mu\|_*\le 1\}$ is weak* compact by Alaoglu's theorem. That is, dual norm bounded sequences 
have weak* convergent subsequences.

\subsection{ Gaussian measures}
We'll say that $\eta$ is a centered, nondegenerate Gaussian measure on $X$ provided it is a Borel probability measure such that the pushforward of $\eta$ by any $\mu\in X^*\setminus\{0\}$ is a centered, nondegenerate Gaussian measure on $\R$. That is,
$$
\int_Xg(\mu(x))d\eta(x)=\int_{\R}g(y)\frac{e^{-\frac{y^2}{2q}}}{\sqrt{2\pi q}}dy
$$
for some $q> 0$ and all bounded and continuous $g: \R\rightarrow \R$. For convenience, we'll simply refer to a centered, nondegenerate Gaussian measure as a {\it Gaussian measure}. Below we will recall some basic properties of Gaussian measures for the purposes of this paper, which can be found in
\cite{MR1642391,eldredge2016analysis,hairer2009introduction}. 

\par Let $\eta$ be a Gaussian measure on $X$.  It is known that $\eta$ has a finite second moment
$$
\int_{X}\|x\|^2d\eta(x)<\infty.
$$
We'll write 
$$
(\mu,\nu)=\int_X\mu(x)\nu(x)d\eta(x)
$$
for the inner product between $\mu$ and $\mu$ in $L^2(X,\eta)$ and 
$$
\|\mu\|_{L^2}=(\mu,\mu)^{1/2}.
$$
Note that if $\mu\in X^*$ and $x\in X$, $|\mu(x)|\le \|\mu\|_*\|x\|.$ It follows that 
$$
\|\mu\|_{L^2}\le \|\mu\|_*\left(\int_{X}\|x\|^2d\eta(x)\right)^{1/2}.
$$
Therefore,  $X^*\subset L^2(X,\eta).$

\par As in the introduction, we denote 
$$
H=\text{closure of $X^*$ in the $L^2(X,\eta)$ norm}.
$$
We can think of $H$ being linear functionals on $X$ which are merely square integrable with respect to $\eta$. 
With this choice of Hilbert space $H$, 
\be\label{ContEmbedML2}
X^*\subset H
\ee
is a dense subspace.  Moreover, this embedding is compact.

\par By our definition of a Gaussian measure, if $\mu\in X^*\setminus\{0\}$, then $(\mu,\mu)>0$. Therefore, if $\mu_1,\mu_2\in X^*$ are equal $\eta$ almost everywhere, they must agree everywhere on $X$. 
The following lemma is a consequence of this observation. 
\begin{lem}\label{TechLEm1}
Suppose $(\mu^k)_{k\in \N}$ is a bounded sequence in $X^*$ which converges weakly in $H$ to $\xi$. Then $(\mu^k)_{k\in \N}$ converges weak*, and its limit agrees $\eta$ almost everywhere with $\xi$. 
\end{lem}
\begin{proof}
Choose a subsequence $(\mu^{k_j})_{j\in \N}$ which converges weak* to some $\mu\in X^*$.  Note that for a given $\zeta\in H$, 
$$
|\mu^k(x)\zeta(x)|\le c\|x\||\zeta(x)|
$$
for some $c$ independent of $k\in\N$ and $x\in X$. Observe that the right hand side above is in $L^1(X,\eta)$.  Dominated convergence then gives 
$$
\lim_{j\rightarrow\infty}\int_{X} \mu^{k_j}(x)\zeta(x)d\eta(x)=\int_{X}\mu(x)\zeta(x)d\eta(x).
$$
It follows that $(\mu^{k_j})_{j\in \N}$ converges weakly to $\mu$ in $H$. As a result, $\mu=\xi$ almost everywhere.   If ($\mu^k)_{k\in \N}$ has another weak* subsequential limit $\tilde\mu$, then $\mu(x)=\tilde \mu(x)$ for $\eta$ almost $x\in X$.  Therefore, $\mu\equiv \tilde\mu$ and  $(\mu^k)_{k\in \N}$ converges to $\mu$ since this limit is independent of the subsequence. 
\end{proof}

\subsection{The mapping $\Jay$}
We now consider the linear mapping $\Jay: H\rightarrow X$ defined by the formula 
\be\label{jayFormula}
\Jay\xi=\int_{X}x\xi(x)d\eta(x).
\ee
Observe that this Bochner integral is a well defined element of $X$. Indeed, since $\xi$ is the $L^2(X,\eta)$ limit of a sequence of continuous functions and since $X$ is separable, the mapping $x\mapsto x\xi(x)$ from $X$ into $X$ is strongly measurable; this can be seen as a consequence of Pettis' theorem (Chapter V section 4 of \cite{MR617913}). Moreover, $x\mapsto \|x\xi(x)\|$ is clearly in $L^1(X,\mu)$.

\par 
A basic assertion regarding $\Jay$ is as follows. 
\begin{prop}
$(i)$ For $\mu\in X^*$ and $\xi\in H$,
\be\label{jayidentity}
\langle \mu,\Jay\xi\rangle=(\mu,\xi).
\ee
$(ii)$ $\Jay: H\rightarrow X$ is continuous and injective. 
\end{prop}
\begin{proof}
$(i)$ As $\Jay\xi$ is the Bochner integral \eqref{jayFormula},
$$
\langle \mu,\Jay\xi\rangle=\left\langle\mu, \int_{X}x\xi(x) d\eta(x)\right\rangle=\int_X\mu(x)\xi(x)d\eta(x).
$$
$(ii)$ Since
$$
\|\Jay\xi\|\le \left(\int_{X}\|x\|^2d\eta(x)\right)^{1/2}\|\xi\|_{L^2}
$$
for $\xi\in H$, $\Jay$ is bounded. And if $\Jay\xi=0\in X$, then $(\mu,\xi)=0$ for each $\mu\in X^*$. Since 
$X^*$ is dense in $H$, $(\mu,\xi)=0$ for each $\mu\in H$. That is, $\xi=0\in H$.
\end{proof}
\begin{rem}
It is also routine to verify that for $\mu\in X^*$, $\Jay\mu$ is the Fr\'echet derivative of 
$X^*\ni \nu\mapsto \frac{1}{2}\|\nu\|^2_{L^2}$ at $\mu$.
\end{rem}

\section{$N$-player games}\label{NplayGameSect}
The primary goal of this section is to prove Theorem \ref{ThmOne}. To this end, we suppose $K_1,\dots, K_N$ are each nonempty, convex, and compact subsets of $X^*$ and set $K=K_1\times \cdots \times K_N$. We will assume $F_j: K\rightarrow \R$ is weak* continuous and that
$\nu_j\mapsto F_j(\nu_j,\mu_{-j})$ is convex for each $\mu\in K$ and $j=1,\dots, N$.

\par First let us recall that a Nash equilibrium exists. 
\begin{prop}
$F_1,\dots, F_N$ has a Nash equilibrium. 
\end{prop}
\begin{proof}
By our assumptions, the mapping from $K$ into $2^K$ 
$$
K\ni \mu\mapsto \text{argmin}\left\{ \sum^N_{j=1}F_j(\nu_j,\mu_{-j}): \nu\in K\right\}
$$
has nonempty and convex images. Moreover, it is routine to check that the graph of this mapping is closed. 
It follows from the Fan-Glicksberg theorem \cite{MR47317,MR46638}, that
there is a fixed point
$$
\mu\in \text{argmin}\left\{ \sum^N_{j=1}F_j(\mu_j,\nu_{-j}): \nu\in K\right\},
$$
which is also a Nash equilibrium of $F_1,\dots, F_N$.  
\end{proof}
\begin{rem}
Fan \cite{MR47317} and Glicksberg \cite{MR46638} independently generalized Kakutani's fixed point theorem \cite{MR4776} to locally convex spaces. 
\end{rem}

\subsection{Monotonicity of $F_1,\dots, F_N$}
Recall that $F_1,\dots, F_N$ is monotone provided \eqref{newFmonotone} holds.  There is also a simple sufficient condition for monotonicity as detailed in the proposition below. 
\begin{prop}\label{OneMonImpliesOther}
Suppose for each $\mu,\nu\in K$, 
\be\label{FeyeMonotoneTwo}
\sum^N_{j=1}(F_j(\mu)+F_j(\nu))\ge \sum^N_{j=1}(F_j(\nu_j,\mu_{-j})+F_j(\mu_j,\nu_{-j})).
\ee
Then $F_1,\dots, F_N$ is monotone. 
\end{prop} 
\begin{proof}
Suppose $x_j\in \partial_{\mu_j}F_j(\mu)$ and $y_j\in \partial_{\mu_j}F_j(\nu)$. Then 
$$
F_j(\nu_j,\mu_{-j})\ge F_j(\mu)+\langle \nu_j-\mu_j,x_j\rangle
$$
and 
$$
F_j(\mu_j,\nu_{-j})\ge F_j(\nu)+\langle \mu_j-\nu_j,y_j\rangle
$$
for $j=1,\dots, N$. Adding these inequalities yields 
\begin{align}
\sum^N_{j=1}(F_j(\nu_j,\mu_{-j})+F_j(\mu_j,\nu_{-j}))&\ge \sum^N_{j=1}(F_j(\mu)+F_j(\nu)) -\sum^N_{j=1}\langle \mu_j-\nu_j,x_j-y_j\rangle.
\end{align}
Using \eqref{FeyeMonotoneTwo} gives 
$$
\sum^N_{j=1}\langle \mu_j-\nu_j,x_j-y_j\rangle\ge 0.
$$
\end{proof}
\begin{rem}
The proof above shows 
\be\label{MonotonicityIneq}
\sum^N_{j=1}\langle \mu_j-\nu_j,x_j-y_j\rangle
\ge \sum^N_{j=1}\left(F_j(\mu)+F_j(\nu)\right)-\sum^N_{j=1}(F_j(\nu_j,\mu_{-j})+F_j(\mu_j,\nu_{-j}))
\ee
whenever $x_j\in \partial_{\mu_j}F_j(\mu)$ and $y_j\in \partial_{\mu_j}F_j(\nu)$ for $j=1,\dots, N$.
\end{rem}
\begin{cor}\label{twoPlayerZeroSumMon}
If $N=2$ and $F_1+F_2\equiv 0$, then $F_1, F_2$ is monotone. That is, two-person zero-sum games are monotone. 
\end{cor}
\begin{proof}
For $\mu,\nu\in K$,
$$
\sum^2_{j=1}(F_j(\mu)+F_j(\nu))=\sum^2_{j=1}F_j(\mu)+\sum^2_{j=1}F_j(\nu)=0+0=0,
$$
and 
\begin{align}
\sum^2_{j=1}(F_j(\nu_j,\mu_{-j})+F_j(\mu_j,\nu_{-j}))&=(F_1(\nu_1,\mu_2)+F_1(\mu_1,\nu_2))+(F_2(\mu_1,\nu_2)+F_2(\nu_1,\mu_2))\\
&=(F_1(\nu_1,\mu_2)+F_2(\nu_1,\mu_2))+(F_1(\mu_1,\nu_2))+F_2(\mu_1,\nu_2))\\
&=0+0\\
&=0.
\end{align}
\end{proof}
\par We also note that monotonicity can be verified somewhat more easily in the model case. 
\begin{prop}
Suppose $S$ is a compact metric space, $f_j: S^N\rightarrow \R$ is continuous, and set
\be\label{modelFeye}
F_j(\mu_1,\dots, \mu_N)=\int_{S^N}f_j(s)d\mu_1(s_1)\cdots d\mu_N(s_N)
\ee
for $\mu\in\Pee(S)^N$ and $j=1,\dots, N$. Then $F_1,\dots, F_N$ satisfies \eqref{FeyeMonotoneTwo} if and only 
\be\label{lilFeyeMonotone}
 \sum^N_{j=1}(f_j(s)+f_j(t))\ge \sum^N_{j=1}(f_j(s_j,t_{-j})+f_j(t_j,s_{-j}))
\ee
for all $s,t\in S^N$. 
\end{prop}
\begin{proof}
Suppose \eqref{lilFeyeMonotone} holds and $s,t\in S^N$. If we select $\mu_j=\delta_{s_j}$ and $\nu_j=\delta_{t_j}$ for $j=1,\dots, N$, then \eqref{FeyeMonotoneTwo} is the same inequality as \eqref{lilFeyeMonotone}. Alternatively, suppose \eqref{lilFeyeMonotone} holds and $\mu,\nu\in\Pee(S)^N$. Integrating this inequality against $d\mu_j(s_j)d\nu_j(t_j)$ for $j=1,\dots, N$ leads to \eqref{FeyeMonotoneTwo}.
\end{proof}

\subsection{Flow of mixed strategies}
Assume $\eta$ is a Gaussian measure on $X$.   As previously mentioned, the closure of $X^*$ in $L^2(X,\eta)$ is a Hilbert space $H$ with inner product $(\cdot,\cdot)$. We also will employ the linear mapping $\Jay: H\rightarrow X$ defined in \eqref{jayFormula} and suppose going forward that $F_1,\dots, F_N$ is monotone. 

\par We will now consider the problem of finding a solution $\xi: [0,\infty)\rightarrow H^N$ of the initial value problem
\be\label{mainIVP} 
\begin{cases}
\Jay\dot \xi_j(t)+\partial_{\mu_j}F_j(\xi(t))\ni 0\;\text{a.e. $t\ge 0$}
\\
\xi_j(0)=\mu^0_j
\end{cases}
\ee
for a given $\mu^0\in K$. Here $H^N$ is the $N$-fold product of $H$ endowed with the inner product 
$$
(\mu,\nu):=\sum^N_{j=1}(\mu_j,\nu_j).
$$
Observe that if  
$\xi$ and $\vartheta$ are solutions of the initial value problem \eqref{mainIVP} with possibly distinct initial conditions 
$\xi(0)=\xi^0\in K$ and $\vartheta(0)=\vartheta^0\in K$, then
\begin{align*}
\frac{d}{dt}\frac{1}{2}\sum^N_{j=1}\|\xi_j(t)-\vartheta_j(t)\|^2_{L^2}&=\sum^N_{j=1}(\xi_j(t)-\vartheta_j(t),\dot\xi_j(t)-\dot\vartheta_j(t))\\
&=\sum^N_{j=1}\langle\xi_j(t)-\vartheta_j(t),\Jay\dot\xi_j(t)-\Jay\dot\vartheta_j(t)\rangle\\
&\le 0
\end{align*}
for almost every $t\ge 0$.  As a result, we expect 
$$
\sum^N_{j=1}\|\xi_j(t)-\vartheta_j(t)\|^2_{L^2}\le \sum^N_{j=1}\|\xi_j^0-\vartheta_j^0\|^2_{L^2}
$$
for $t\ge 0$ and the initial value problem \eqref{mainIVP} to generate a contraction semigroup in $H^N$. 

\par In order to establish this, we will 
define a mapping 
$$
A: H^N\rightarrow 2^{H^N}
$$
via 
\be
A\mu:=
\begin{cases}
\displaystyle\Jay^{-1}\left(\partial_{\mu_1}F_1(\mu)\right)\times\cdots\times \Jay^{-1}\left(\partial_{\mu_N}F_N(\mu)\right), &\quad\mu\in {\cal D}\\
\displaystyle\emptyset, &\quad \mu\not\in {\cal D}
\end{cases}
\ee
for $\mu\in H^N$.  Here we recall that ${\cal D}$ is defined in \eqref{DeeSet} and emphasize that $\sigma\in A\mu$ provided $\mu\in K$ and $\Jay\sigma_j\in \partial_{\mu_j}F_j(\mu)$ for $j=1,\dots, N$.

\par We will show that $A$ is maximally monotone. We recall this means that $A$ is monotone on $H^N$ and that there is no other monotone operator on $H^N$ whose graph properly includes the graph of $A$.
\begin{lem}
A is maximally monotone.  
\end{lem}
\begin{proof}
Suppose $\sigma\in A\mu$ and $\tilde\sigma\in A\tilde\mu$. Then
\begin{align}
(\mu-\tilde\mu,\sigma-\tilde\sigma)&=\sum^N_{j=1}(\mu_j-\tilde\mu_j,\sigma_j-\tilde\sigma_j)
=\sum^N_{j=1}\langle\mu_j-\tilde\mu_j,\Jay\sigma_j-\Jay\tilde\sigma_j\rangle\ge 0
\end{align}
as $F_1,\dots,F_N$ is monotone. Thus, $A$ is monotone.

\par In order to show that $A$ is maximal, we will appeal to Minty's lemma \cite{MR169064}. That is, it suffices to show for each $\sigma\in H^N$, there is $\mu\in K$ with $\mu+A\mu\ni \sigma.$ Note that $\mu$ is a solution if and only if 
$$
\Jay(\mu_j-\sigma_j)+\partial_{\mu_j}F_j(\mu)\ni 0
$$ 
for $j=1,\dots, N$. Furthermore, it is routine to check that $\mu$ is the desired solution if and only if $\mu\in \Phi(\mu)$, where $\Phi:K\mapsto 2^K$ is defined 
$$
\Phi(\mu):=\text{argmin}\left\{ \sum^N_{j=1}\langle \nu_j,\Jay(\mu_j-\sigma_j)\rangle+F_j(\nu_j,\mu_{-j}): \nu\in K\right\}
$$
for $\mu\in K$.

\par Note that 
$$
K\ni\nu\mapsto \sum^N_{j=1}\langle \nu_j,\Jay(\mu_j-\sigma_j)\rangle+F_j(\nu_j,\mu_{-j})
$$
is weak* continuous for each $\mu\in K$. Since $K$ is weak* compact, this function has a minimum. And as this function is convex, its set of minima is convex. Thus, $\Phi(\mu)$ is nonempty and convex. The continuity  of $F_1,\dots, F_N$ and of $\Jay$ also imply that the graph of $\Phi$ is closed. Therefore, there is $\mu\in K$ such that $\mu\in \Phi(\mu)$ by the Fan-Glicksberg theorem \cite{MR47317,MR46638}.  
\end{proof}

\par We can now verify the initial value problem  \eqref{mainIVP} has a solution whose Ces\`aro mean converges to a Nash equilibrium. 
\begin{proof}[Proof of Theorem \ref{ThmOne}]
We've established that $A$ is maximally monotone. Since $\mu^0$ belongs to the domain of $A$, Theorem 3.1 of \cite{MR0348562} implies there is a unique absolutely continuous solution $\xi:[0,\infty)\rightarrow H^N$ of the equation 
\be
\begin{cases}
\dot\xi(t)+A\xi(t)\ni 0 \text{ a.e. $t\ge 0$}\\
\xi(0)=\mu^0.
\end{cases}
\ee
Moreover, $\xi$ is Lipschitz continuous and $A\xi(t)\neq 0$ for each $t\ge 0$. It follows that $\xi(t)\in {\cal D}$ for $t\ge 0$ and that 
$$
\Jay\dot\xi_j(t)+\partial_{\mu_j}F_j(\xi(t))\ni 0\; \text{a.e. $t\ge 0$}
$$
for $j=1,\dots, N$.  Consequently, $\xi$ is a solution of the initial value problem \eqref{mainIVP} as claimed. 

\par The limit 
$$
\mu_j:=\lim_{t\rightarrow\infty}\frac{1}{t}\int^t_0\xi_j(s)ds
$$
exists weakly in $H$ by the Baillon-Br\'ezis theorem \cite{MR394328} for each $j=1,\dots, N$ and satisfies $0\in A\mu$. That is, $\partial_{\mu_j}F_j(\mu)=0$ for $j=1,\dots, N$, so $\mu$ is a Nash equilibrium of $F_1,\dots, F_N$. By Lemma \ref{TechLEm1}, this limit  also exists weak*. 
\end{proof}

\section{Symmetric games}\label{SymmetricMFGsect}
Suppose now that $K=K_1=\cdots =K_N\subset X^*$ is weak* compact, $F_1,\dots, F_N: K^N\rightarrow \R$ is continuous, and $K\ni \nu_j\mapsto F_j(\nu_j,\mu_{-j})$ is convex for each $\mu\in K^N$ and $j=1,\dots,N$. We will say that $F_1,\dots, F_N$ is {\it symmetric} provided
\be\label{FoneFNsymmetric}
F_i\underbrace{(\mu,\dots,\mu,\nu,\mu,\dots,\mu)}_{\text{$\nu$ is in the $i$th argument of $F_i$}}=F_j\underbrace{(\mu,\dots,\mu,\nu,\mu,\dots,\mu)}_{\text{$\nu$ is in the $j$th argument of $F_j$}}
\ee
for all $i,j=1,\dots, N$ and all $\mu,\nu\in K$.   With these assumptions, it can be shown that $F_1,\dots, F_N$ has a symmetric Nash equilibrium $(\mu,\dots,\mu)\in K^N.$  As we only need to find $\mu\in K$ such that  
$$
F_1(\mu,\dots, \mu)\le F_1(\nu,\mu,\dots,\mu)\;\text{for all $\mu\in K$},
$$
we can employ a simpler approximation method than discussed above. 

\par The theorem we present below will also apply to mean field equilibria which we recall are $\mu\in \Pee(S)$ that satisfy 
$$
\int_{S}f(s,\mu)d\mu(s)\le \int_{S}f(s,\mu)d\nu(s)
$$
for each $\nu \in \Pee(S)$. For static mean field games, we'll always assume $S$ is a compact metric space and $f:S\times\Pee(S)$ is continuous; here $S\times\Pee(S)$ is endowed with the product topology form the metric on $S$ and the weak* topology on $\Pee(S)$.  The key monotonicity condition that will be needed is  
\be\label{littleeffMonotone}
\int_{S}(f(s,\mu)-f(s,\nu))d(\mu-\nu)(s)\ge 0
\ee
for $\mu,\nu\in\Pee(S)$.

\par In order to address both scenarios, we will consider a weak* continuous 
$G:K\times K\rightarrow \R$
such that
$$
K\ni \nu\mapsto G(\mu,\nu) \text{ is convex for each $\nu\in K$}.
$$
We'll also say $\mu\in K$ is an {\it equilibrium} for $G$ provided
$$
G(\mu,\mu)\le G(\mu,\nu)
$$
for all $\nu\in K$. Moreover, $\mu$ is an equilibrium if and only if $0\in \partial_\nu G(\mu,\mu)$.  Here 
\be
\partial_{\nu}G(\mu,\mu)=\{x\in X:G(\nu,\mu)\ge G(\mu,\mu)+\langle\nu-\mu,x\rangle \;\text{for $\nu\in K$} \}.
\ee

\par It is not hard to see that $G$ has an equilibrium. 
\begin{lem}
There exists an equilibrium for $G$. 
\end{lem}
\begin{proof}
Note that $\mu$ is an equilibrium for $G$ if and only if $\mu$ is a fixed point of the mapping from $K$ into $2^K$ given by 
$$
K\ni\sigma \mapsto \text{argmin}\left\{G(\sigma,\nu): \nu\in K\right\}.
$$
With the continuity and convexity assumptions made on $G$, it is routine to show this mapping has a fixed point by applying the 
Fan-Glicksberg theorem \cite{MR47317,MR46638}.
\end{proof}

\subsection{Monotonicity of $G$}
 We will say that $G$ is {\it monotone} provided
\be\label{EffMon}
\langle \mu-\nu , x-y\rangle \ge 0
\ee
whenever $x\in   \partial_\nu G(\mu,\mu)$ and $y\in  \partial_\nu G(\nu,\nu)$.
As in the case of $N$-player games,  it sometimes is useful to identify a simple sufficient condition for monotonicity.  
\begin{lem}\label{GmonCond}
Suppose
\be\label{EffMonTwo}
G(\mu,\mu)+G(\nu,\nu)\ge G(\mu,\nu)+G(\nu,\mu)
\ee
for all $\mu,\nu\in K$. Then $G$ is monotone. 
\end{lem}
\begin{proof}
Let $x\in   \partial_\nu G(\mu,\mu)$ and $y\in  \partial_\nu G(\nu,\nu)$. Then 
$$
G(\mu,\nu)\ge G(\mu,\mu)+\langle x, \nu-\mu\rangle 
$$
and 
$$
G(\nu,\mu)\ge G(\nu,\nu)+\langle y, \mu-\nu\rangle. 
$$
Adding these inequalities gives
$$
G(\mu,\nu)+G(\nu,\mu)\ge G(\mu,\mu)+G(\nu,\nu)-\langle x-y, \mu-\nu\rangle.
$$
In view of \eqref{EffMonTwo}, we conclude that \eqref{EffMon} holds. 
\end{proof}

\begin{ex} One of the model cases occurs when $F_1,\dots, F_N$ is symmetric.  Here the relevant $G$ function is  
$$
G(\mu,\nu)=F_1(\nu,\mu,\dots, \mu)
$$
for $\mu,\nu\in K$.  It is routine to verify that if the collection $F_1,\dots, F_N$ is additionally monotone, then $G$ is monotone. Using Lemma \ref{GmonCond}, it is also possible to show that a sufficient condition for the monotonicity of $G$ is 
$$
F_1(\mu,\dots, \mu)+F_1(\nu,\dots, \nu)\ge F_1(\nu,\mu,\dots, \mu)+F_1(\mu,\nu,\dots, \nu)
$$
for $\mu,\nu\in K$. 
\end{ex}

\par Let us also briefly consider the case of a static mean field game $f: S\times \Pee(S)\rightarrow \R$. 
Here 
$$
G(\mu,\nu)=\int_Sf(s,\mu)d\nu(s)\quad (\mu,\nu\in \Pee(S))
$$
is monotone provided \eqref{littleeffMonotone} holds.  We will look at a few examples of $f$ below which satisfy this monotonicity condition.
\begin{ex}
$f(x,\mu)=\varphi(x)$ for any continuous $\varphi:S\rightarrow \R$.  This example clearly satisfies \eqref{littleeffMonotone}. Note that any minimizer $s\in S$ of $\varphi$ corresponds to the mean field equilibrium $\mu=\delta_s$. 
\end{ex}
\begin{ex}
Consider 
$$
f(s,\mu)=\displaystyle\int_Sk(s,t)d\mu(t)
$$
for any continuous, symmetric, and nonnegative definite kernel $k:S\times S\rightarrow \R$. That is, for any $s_1,\dots, s_N\in S$ and $c_1,\dots, c_N\in \R$, 
$$
\sum_{i,j=1}^Nk(s_i,s_j)c_ic_j\ge 0.
$$
It is routine to check that these assumptions imply $f$ satisfies \eqref{littleeffMonotone}.
\end{ex}
\begin{ex}
Suppose $\ell$ is a positive Borel measure on $S$ with $\ell(S)>0$. Set
$$
f(s,\mu)=\psi(\rho(s)), \quad (s\in S)
$$
whenever $d\mu/d\ell=\rho$.  Here $\psi$ is an increasing function on $\R$.  Of course, $f$ will not in general be continuous. However, if $d\nu=\sigma /d\ell $, then 
$$
\int_S(f(s,\mu)-f(s,\nu))d(\mu-\nu)(s)=\int_S(\psi(\rho(s))-\psi(\sigma(s)))(\rho(s)-\sigma(s))d\ell(s)\ge 0.
$$
It is also evident that the density
$$
\rho(s)=\frac{1}{\ell(S)}\quad (s\in S)
$$
defines a mean field equilibrium. 
\end{ex}
\subsection{Another flow of mixed strategies}
Suppose $G$ is monotone and $\eta$ is a Gaussian measure on $X$.  We now will show how to approximate an equilibrium of $G$.  For a given $\mu^0\in K$, we consider the initial value problem: find an absolutely continuous $\zeta: [0,\infty)\rightarrow H$ such that 
\be\label{MeanFieldIVP}
\begin{cases}
\Jay\dot\zeta(t)+\partial_{\nu}G(\zeta(t),\zeta(t))\ni 0\quad \text{a.e. $t\ge 0$}\\
\;\zeta(0)=\mu^0.
\end{cases}
\ee
Here $\Jay: H\rightarrow X$ is the linear mapping \eqref{jayFormula}.

\par Note that if $\zeta$ and $\chi$ are two solutions of \eqref{MeanFieldIVP} perhaps with distinct initial conditions $\zeta(0)=\zeta^0$ and $\chi(0)=\chi^0$ then
\begin{align}
\frac{d}{dt}\frac{1}{2}\|\zeta(t)-\chi(t)\|^2_{L^2}&=(\zeta(t)-\chi(t),\dot\zeta(t)-\dot\chi(t))\\
&=\langle\zeta(t)-\chi(t),\Jay\dot\zeta(t)-\Jay\dot\chi(t)\rangle\\
&\le 0
\end{align}
for almost every $t\ge 0$. The inequality above follows by our monotonicity assumption on $G$.  Therefore, 
$$
\|\zeta(t)-\chi(t)\|_{L^2}\le \|\zeta^0-\chi^0\|_{L^2}
$$
for $t\ge 0$, and we expect \eqref{MeanFieldIVP} to generate a contraction semigroup in $H$. 

\par In order to establish these claims, we will introduce the operator $B:H\rightarrow 2^H$ defined by 
$$
B\mu:=
\begin{cases}
\Jay^{-1}\left(\partial_\nu G(\mu,\mu)\right), \quad &\mu\in {\cal C}\\
\emptyset, \quad &\mu\not\in {\cal C}.
\end{cases}
$$
Here
$$
{\cal C}=\Big\{\mu\in K: \Jay^{-1}\left(\partial_\nu G(\mu,\mu)\right)\neq \emptyset\Big\},
$$
and we note that $\sigma\in B\mu$ provided $\mu\in K$ and $\Jay\sigma\in \partial_\nu G(\mu,\mu)$.
\begin{lem}
$B$ is maximally monotone. 
\end{lem}
\begin{proof}
Let $\sigma\in B\mu$ and $\tilde\sigma\in B\tilde\mu$. Since $G$ is monotone 
\be
(\sigma-\tilde\sigma,\mu-\tilde\mu)=\langle\sigma-\tilde\sigma,\Jay\mu-\Jay\tilde\mu\rangle\ge 0.
\ee
Thus, $B$ is monotone.  

\par Suppose that $\sigma\in H$. We claim that there is $\mu\in K$ for which 
$\mu+B\mu\in\sigma$.   This is equivalent to the condition $\mu\in \Psi(\mu)$, where $\Psi: K\rightarrow 2^K$ is defined
$$
\Psi(\mu)=\text{argmin}\left\{ \langle\nu, \Jay(\mu-\sigma)\rangle +G(\mu,\nu): \nu\in K\right\}.
$$
It is routine to check that $\Psi(\mu)$ is nonempty and convex for each $\mu\in K$ and that the graph of $\Psi$ is closed. 
It then follows from the Fan-Glicksberg theorem \cite{MR47317,MR46638} that $\Psi$ has a fixed point. We conclude that 
$B$ is maximal by Minty's lemma. 
\end{proof}
Our second approximation theorem is as follows. 
\begin{thm}\label{ThmTwo}
Suppose $\mu^0\in {\cal C}$.  There is a unique absolutely continuous $\zeta: [0,\infty)\rightarrow H$ with $\zeta(t)\in {\cal C}$ for all $t\ge 0$ that satisfies \eqref{MeanFieldIVP}.  Furthermore, $\zeta$ is Lipschitz continuous and
$$
\frac{1}{t}\int^t_0\zeta(s)ds
$$
converges weak* to an equilibrium of $G$ as $t\rightarrow\infty$.
\end{thm}
\begin{proof}
By Theorem 3.1 of \cite{MR0348562}, there is a unique absolutely continuous $\zeta: [0,\infty)\rightarrow H$ such that 
\be
\begin{cases}
\dot\zeta(t)+B\zeta(t)\ni 0 \text{ a.e. $t\ge 0$} \\
\zeta (0)=\mu^0.
\end{cases}
\ee
Moreover, $\zeta$ is Lipschitz continuous and $B\zeta(t)\neq \emptyset$ for $t\ge 0$. By design, $\zeta$ also solves the initial value problem \eqref{MeanFieldIVP}. In view of the Baillon-Br\'ezis theorem \cite{MR394328}, the limit 
$$
\mu:=\lim_{t\rightarrow\infty}\frac{1}{t}\int^t_0\zeta(s)ds
$$
exists weakly in $H$ and satisfies $0\in B\mu$. Thus, $\mu$ is an equilibrium of $G$.  The above limit also occurs weak* by Lemma \ref{TechLEm1}. 
\end{proof}

\begin{ex}
Suppose $F_1,\dots, F_N$ is monotone and symmetric. In order to approximate a symmetric Nash equilibrium,  we can use a solution $\zeta:[0,\infty)\rightarrow H$ of
\be
\begin{cases}
\Jay\dot\zeta(t)+\partial_{\mu_1}F_1(\zeta(t),\dots,\zeta(t))\ni 0\quad \text{a.e. $t\ge 0$}\\
\;\zeta(0)=\mu^0.
\end{cases}
\ee
According to Theorem \ref{ThmTwo}, there is a solution whose Ces\`aro mean converges weak* to a symmetric Nash equilibrium $\mu$  provided that  
$$
\Jay\sigma\in\partial_{\mu_1}F_1(\mu^0,\dots,\mu^0)
$$
for some $\sigma\in H$. 
\end{ex}
\begin{ex}
Let us consider a static mean field game $f: S\times\Pee(S)\rightarrow \R$ such that \eqref{littleeffMonotone} holds. By Theorem  \ref{ThmTwo}, there is a path $\zeta:[0,\infty)\rightarrow H$ which satisfies 
\be\label{MFGflow}
\begin{cases}
\langle \nu-\zeta(t),\Jay\dot\zeta(t)+ f(\cdot,\zeta(t))\rangle\ge 0\quad \text{for a.e. $t\ge 0$ and all $\nu\in \Pee(S)$}\\
\;\zeta(0)=\mu^0,
\end{cases}
\ee
provided $\mu^0\in K$ and there is $\sigma\in H$ such that
$$
\langle\nu-\mu^0,-\Jay\sigma+ f(\cdot,\mu^0)\rangle\ge 0 \quad\text{for all $\nu\in \Pee(S)$}.
$$ 
Moreover, the Ces\`aro mean of $\zeta$ converges weak* to a mean field equilibrium as $t\rightarrow \infty$. \end{ex}

\appendix

\section{Finite action sets}
We will consider a particular Gaussian measure on $X=C(S)$ with
$$
S=\{s_1,\dots,s_m\}.
$$
These considerations will be used to show how our general theory applies to games with finite action sets. In particular, 
we will informally argue below that the abstract flows considered in this paper reduce to much simpler flows on finite dimensional
spaces. 

\par To this end, it will be convenient to define $e_1,\dots, e_m: S\rightarrow \R$ via
$$
e_j(s_i)=\delta_{ij}\text{ for $i,j=1,\dots,m$}.
$$
This allows us to represent each $f\in C(S)$ and $\mu\in M(S)$ as
$$
f=\sum^m_{j=1}f(s_j)e_j\quad\text{and}\quad\mu=\sum^m_{j=1}\mu(e_j)\delta_{s_j}.
$$
These representations can be used to verify that $C(S)$ is isometrically isomorphic to $\R^m$ endowed with the $\infty$-norm
and that $M(S)$ is isometrically isomorphic to $\R^m$ endowed with the $1$-norm. It is also routine is to show that if $\mu\in \Pee(S)$, then
$$
(\mu(e_1),\dots,\mu(e_m))\in \Delta_m.
$$

\par  We will consider the Borel probability measure $\gamma$ on $C(S)$ defined as
\be\label{GammaMeas}
\int_{C(S)}hd\gamma=\int_{\R^m}h\left(\sum^m_{j=1}x_je_j\right)\frac{1}{(2\pi)^{m/2}}e^{-\frac{1}{2}|x|^2}dx
\ee
for continuous and bounded $h: C(S)\rightarrow \R$. 
\begin{prop}
$\gamma$ is a Gaussian measure. Moreover, 
\be\label{FiniteDimInner}
(\mu,\nu)=\sum_{j=1}^m\mu(e_j)\nu(e_j)
\ee
for $\mu,\nu\in M(S)$, and
\be\label{JayFiniteDim}
\Jay\left(\sum^m_{j=1}\nu(e_j)\delta_{s_j}\right)=\sum_{j=1}^m\nu(e_j)e_j
\ee
for $\nu\in M(S)$.
\end{prop}
\begin{proof}
Suppose $g: \R\rightarrow \R$ is bounded and continuous and $c=(c_1,\dots,c_m)\in \R^m\setminus\{0\}$. Observe that 
\begin{align}
\int_{C(S)}g\left(\sum^m_{j=1}c_j\delta_{s_j}\right)d\gamma&=\int_{\R^m}g\left(\sum^m_{j=1}c_jx_j\right)\frac{1}{(2\pi)^{m/2}}e^{-\frac{1}{2}|x|^2}dx\\
&=\int_{\R^m}g\left(c\cdot x\right)\frac{1}{(2\pi)^{m/2}}e^{-\frac{1}{2}|x|^2}dx\\
&=\int_{\R^m}g\left(|c|x_1\right)\frac{1}{(2\pi)^{m/2}}e^{-\frac{1}{2}|x|^2}dx\\
&=\int_{\R}g\left(|c|x_1\right)\frac{1}{(2\pi)^{1/2}}e^{-\frac{1}{2}x_1^2}dx_1\\
&=\int_{\R}g\left(y\right)\frac{1}{(2\pi)^{1/2}|c|}e^{-\frac{1}{2|c|^2}y^2}dy.
\end{align} 
Thus, $\gamma$ is Gaussian measure.  Also note
\begin{align}
(\mu,\nu)&=\sum_{i,j=1}^m\mu(e_i)\nu(e_j)\int_{C(S)}\delta_{s_i}\delta_{s_j}d\eta \\
              &=\sum_{i,j=1}^m\mu(e_i)\nu(e_j)\int_{\R^m}x_ix_j\frac{e^{-\frac{1}{2}x\cdot x}}{(2\pi)^{m/2}}dx \\
              &=\sum_{i,j=1}^m\mu(e_i)\nu(e_j)\delta_{ij}\\
              &=\sum_{j=1}^m\mu(e_j)\nu(e_j).
\end{align}
This verifies \eqref{FiniteDimInner}.  Formula \eqref{JayFiniteDim} follows from \eqref{FiniteDimInner} and the identity \eqref{jayidentity}.
\end{proof}

\par {\bf Finite action spaces for $N$-player games}.   Suppose $F_j:\Pee(S)^N\rightarrow \R$ is continuous and that
$\nu_j\mapsto F_j(\nu_j,\mu_{-j})$ is convex for each $\mu\in \Pee(S)^N$ and $j=1,\dots, N$. We wish to express the system 
\be\label{xiSystem}
\Jay \dot \xi_j(t)+\partial_{\mu_j}F_j(\xi(t))\ni 0
\ee
$(j=1,\dots, N)$ more concretely.  With this goal in mind, we set 
$$
g_j(x_1,\dots,x_N):=F_j\left(\sum^m_{k=1}x_{1,k}\delta_{s_k},\dots,\sum^m_{k=1}x_{N,k}\delta_{s_k}\right)
$$
for $x_i=(x_{i,1},\dots, x_{i,m})\in \Delta_m$ and $i=1,\dots, N$. We note $g_j$ is continuous and that 
$y_j\mapsto g_j(y_j,x_{-j})$ is convex for each $x\in\Delta_m^N$ and $j=1,\dots, N$.  

\par For a given $x\in \Delta_m^M$ and $j=1,\dots, N$, set
$$
\partial_{x_j}g_j(x):=\Big\{z_j\in \R^m: g_j(y_j,x_{-j})\ge g_j(x)+ z_j\cdot (y_j-x_j)\;\text{for all $y_j\in \Delta_m$}\Big\}.
$$
It is not hard to see that if $\mu_i=\sum^m_{k=1}x_{i,k}\delta_{s_k}$ for $x_i\in \Delta_m$ and $i=1,\dots, N$, then
$$
\sum^N_{k=1}z_{j,k}e_k\in \partial_{\mu_j}F_j(\mu)\text{ if and only if  } z_j\in \partial_{x_j}g_j(x).
$$
\par It follows that the system \eqref{xiSystem} is equivalent to 
\be\label{giSystem}
\dot u_j(t)+\partial_{x_j}g_j(u(t))\ni 0
\ee
($j=1,\dots, N$) for $u:[0,\infty)\rightarrow \Delta_m^N$. That is 
$$
\xi_j(t)=\sum^m_{k=1}u_{j,k}(t)\delta_{s_k}
$$
would solve \eqref{xiSystem} and vice versa.  We finally note that the collection $F_1,\dots, F_N$ is monotone if and only if 
$$
\sum^N_{j=1}( x_j-y_j)\cdot (\partial_{x_j}g_j(x)-\partial_{x_j}g_j(y))\ge 0
$$
for all $x,y\in \Delta_m^N$.
\\
\par {\bf Finite action spaces in mean field games}. If $f:S\times \Pee(S)\rightarrow \R$ is continuous, then
$$
g_j(x):=f\left(s_j,\sum^m_{i=1}x_i\delta_{s_i}\right)\quad (x\in \Delta_m)
$$
is continuous for each $j=1,\dots, m$.  We aim to reinterpret the condition
\be\label{zetaFlow}
\langle \nu-\zeta(t),\Jay\dot\zeta(t)+ f(\cdot,\zeta(t))\rangle\ge 0\text{ for $\nu\in \Pee(S)$}
\ee
in terms of $g_1,\dots, g_m$. 

\par Observe that if  
$$
\nu=\sum^m_{j=1}y_j\delta_{s_j}\quad\text{and}\quad\zeta(t)=\sum^m_{j=1}u_j(t)\delta_{s_j},
$$
then 
\begin{align}
\langle \nu-\zeta(t),\Jay\dot\zeta(t)+ f(\cdot,\zeta(t))\rangle&=
\sum^m_{j=1}(y_j-u_j(t))(\dot u_j(t)+g_j(u(t)))\\
&=(y-u(t))\cdot (\dot u(t)+g(u(t))).
\end{align}
Here we have written $g=(g_1,\dots,g_m)$. As a result, 
$$
(y-u(t))\cdot (\dot u(t)+g(u(t)))\ge 0\text{ for $y\in \Delta_m$}
$$
In particular, this evolution is equivalent to \eqref{zetaFlow}.  Finally, we note that $f$ is monotone in the sense of \eqref{littleeffMonotone} if and only if 
$$
(g(x)-g(y))\cdot (x-y)\ge 0
$$
for $x,y\in \Delta_m$.

\section{An explicit example}
We will work out an example which suggests Ces\`aro mean convergence is the best one may expect from the type of flows considered in this note.  Let us assume that $N=2$ and the cost functions are
$$
\begin{cases}
F_1(x_1,x_2)=3x_{1,1}x_{2,1} +x_{1,2}x_{2,1}+4x_{2,2}x_{1,2}\\
F_2(x_1,x_2)=-3x_{1,1}x_{2,1} -x_{1,2}x_{2,1}-4x_{1,2}x_{2,2}
\end{cases}
$$
for $x_i=(x_{i,1},x_{i,2})\in \Delta_2$ for $i=1,2$.  Note this is a zero-sum game and $\Delta_2\subset \R^2$, where $\R^2$ is equipped with the standard dot product. 
It is not hard to check that the unique Nash equilibrium for $F_1, F_2$ is the pair
$$
((1/2,1/2), (2/3,1/3))\in \Delta_2^2.
$$

\par {\bf Evolution inequalities}. The corresponding flow takes the form 
$$
\left(\begin{array}{cc}
\dot u_{1,1}(t)+ 3u_{2,1}(t)\\
\dot u_{1,2}(t)+ u_{2,1}(t)+4u_{2,2}(t)
\end{array}
\right)\cdot \left(\begin{array}{cc}
z_{1,1}-u_{1,1}(t)\\
z_{1,2}-u_{1,2}(t)\\
\end{array}
\right)\ge 0
$$
and 
$$
 \left(\begin{array}{cc}
\dot u_{2,1}(t)- 3u_{1,1}(t)-u_{1,2}(t)\\
\dot u_{2,2}(t)- 4u_{1,2}(t)
\end{array}
\right)\cdot \left(\begin{array}{cc}
z_{2,1}-u_{2,1}(t)\\
z_{2,2}-u_{2,2}(t)\\
\end{array}
\right)\ge 0
$$
for almost every $t\ge 0$ and each $z_1,z_2\in \Delta_2$.  The unknown is an absolutely continuous path $u:[0,\infty)\rightarrow \Delta_2^2$, where $u(t)=(u_1(t),u_2(t))$.

\par If we put
\be
v_1(t)=u_{1,1}(t),  \; v_2(t)=u_{2,1}(t),\; w_1=z_{1,1},\;\text{and}\; w_2=z_{2,1},
\ee
we can reexpress the above inequalities as
\be
\left(\begin{array}{cc}
\dot v_1(t)+ 3v_2(t)\\
-\dot v_1(t)+ v_2(t)+4(1-v_2(t))
\end{array}
\right)\cdot \left(\begin{array}{cc}
w_1-v_1(t)\\
-(w_1-v_1(t))\\
\end{array}
\right)=(2\dot v_1(t)+6v_2(t)-4)(w_1-v_1(t))\ge 0
\ee
and
\be
\left(\begin{array}{cc}
\dot v_2(t)- 3v_1(t)-(1-v_1(t))\\
-\dot v_2(t)-4(1-v_1(t))\ge 0
\end{array}
\right)\cdot \left(\begin{array}{cc}
w_2-v_2(t)\\
-(w_2-v_2(t))\\
\end{array}
\right)=(2\dot v_2(t)+3-6v_1(t))(w_2-v_2(t))\ge 0.
\ee
Therefore, our initial value problem is equivalent to finding an absolutely continuous pair $v_1,v_2:[0,\infty)\rightarrow [0,1]$ which satisfies 
\be\label{veeOneveeTwoEqn}
\begin{cases}
(\dot v_1(t)+3v_2(t)-2)(w_1-v_1(t))\ge 0\\
(\dot v_2(t)+3/2-3v_1(t))(w_2-v_2(t))\ge 0
\end{cases}
\ee
for each $w_1,w_2\in[0,1]$ and given initial conditions 
\be\label{UzeroVzeroInit}
v_1(0)=v_1^0\in [0,1]\quad \text{and}\quad v_2(0)=v_2^0\in [0,1].
\ee

\par {\bf Solution which parametrizes a circle}. Observe that the solution of the system of ODEs 
$$
\dot v_1(t)+3v_2(t)-2=0\text{ and }\dot v_2(t)+3/2-3v_1(t)=0
$$
subject to the initial conditions \eqref{UzeroVzeroInit} is 
\be\label{CircleVeeParam}
\begin{cases}
v_1(t)=(v_1^0-1/2)\cos(3t)+(2/3-v_2^0)\sin(3t)+1/2
\\
v_2(t)=(v_2^0-2/3)\cos(3t)+(v_1^0-1/2)\sin(3t)+2/3.
\end{cases}
\ee
In particular, this solution parametrizes the circle 
\be\label{CircOfInterest}
(v_1-1/2)^2+(v_2-2/3)^2=(v_1^0-1/2)^2+(v_2^0-2/3)^2
\ee
counterclockwise in the $v_1v_2$ plane. It is easily checked that if
\be\label{smallData}
(v_1^0-1/2)^2+(v_2^0-2/3)^2\le (1/3)^2,
\ee
then $v_1(t),v_2(t)\in [0,1]$ for all $t\ge 0$. In this case, the circular path \eqref{CircleVeeParam} solves \eqref{veeOneveeTwoEqn} and \eqref{UzeroVzeroInit}.
\\
\par {\bf Convergence}. Observe that since the path \eqref{CircleVeeParam} lies on a circle centered at $(1/2,2/3)$, it will not converge to the circle's center as $t\rightarrow\infty$. However, it's plain to see that 
$$
\lim_{t\rightarrow\infty}\frac{1}{t}\int^t_0v_1(s)ds=\frac{1}{2}\quad\text{and}\quad \lim_{t\rightarrow\infty}\frac{1}{t}\int^t_0v_2(s)ds=\frac{2}{3}.
$$
As a result, when \eqref{smallData} holds, the solution of \eqref{veeOneveeTwoEqn} does not converge to the Nash equilibrium of $F_1,F_2$ but its Ces\`aro mean does.  \par It is also possible describe the solution of \eqref{veeOneveeTwoEqn} in the case 
$$
(v_1^0-1/2)^2+(v_2^0-2/3)^2> (1/3)^2.
$$
The corresponding solution will necessarily intersect the boundary of the square. It will then traverse the boundary counterclockwise until it hits the point $(1/2,1)$. For all later times, the solution will traverse  the circle $(v_1-1/2)^2+(v_2-2/3)^2= (1/3)^2$ counterclockwise. As a result, the Ces\`aro mean of this solution will also converge to the Nash equilibrium of $F_1,F_2$. We leave the details to the reader.

\bibliography{MixedNashbib}{}

\bibliographystyle{plain}

\typeout{get arXiv to do 4 passes: Label(s) may have changed. Rerun}

\end{document}